\documentclass[12pt,twoside]{amsart}
\usepackage{amsmath}
\usepackage{amsthm}
\usepackage{amsfonts}
\usepackage{amssymb}
\usepackage{latexsym}
\usepackage{mathrsfs}
\usepackage{amsmath}
\usepackage{amsthm}
\usepackage{amsfonts}
\usepackage{amssymb}
\usepackage{latexsym}
\usepackage{geometry}
\usepackage{dsfont}
\usepackage[dvips]{graphicx}
\usepackage[colorlinks=true,linkcolor=red,citecolor=blue]{hyperref}
\usepackage{color}
\usepackage[all]{xy}

\date{}
\allowdisplaybreaks[4] \footskip=15pt
\renewcommand{\uppercasenonmath}[1]{}

\numberwithin{equation}{section} \theoremstyle{plain}
\usepackage{graphicx,amssymb}
\usepackage[all]{xy}
\usepackage{amsmath}

\allowdisplaybreaks
\usepackage{amsthm}
\usepackage{color}

\theoremstyle{plain}

\theoremstyle{plain}
\newtheorem{theorem}{Theorem}[section]
\newtheorem{proposition}[theorem]{Proposition}
\newtheorem{lemma}[theorem]{Lemma}
\newtheorem{question}[theorem]{Question}

\newtheorem{example}[theorem]{Example}
\newtheorem*{open question}{Open Question}
\newtheorem{definition}[theorem]{Definition}

\theoremstyle{definition}

\theoremstyle{remark}
\newtheorem{remark}[theorem]{Remark}

\newcommand{\Tor}{\mbox{\rm Tor}}

\def\pd{{\rm pd}}

\def\Hom{{\rm Hom}}
\def\Ext{{\rm Ext}}
\def\Tor{{\rm Tor}}

\def\ker{{\rm ker}}
\def\Ker{{\rm Ker}}

\def\Im{{\rm Im}}

\begin{document}
\begin{center}
{\large  \bf On $S$-($h$-)divisible modules and their $S$-strongly flat covers}

\vspace{0.5cm}
Xiaolei Zhang\\

E-mail: zxlrghj@163.com\\
\end{center}


\bigskip
\centerline { \bf  Abstract}
\bigskip
\leftskip10truemm \rightskip10truemm \noindent
It was proved in \cite{BS02} that every $h$-divisible modules admits an strongly flat cover over all integral domains; and every divisible module over an integral domain $R$ admits a strongly flat cover if and only if $R$ is a Matlis domain. 
 In this paper,  we extend these two results to commutative rings with multiplicative subsets.
\\
\vbox to 0.3cm{}

\leftskip0truemm \rightskip0truemm
\bigskip
\section{Introduction}

All rings in this paper are assumed to be commutative with identity, and all modules are unitary. Let $R$ be a ring and $S$ be a multiplicative subset of $R,$ that is, $1\in S$ and $s_1s_2\in S$ for any $s_1\in S, s_2\in S.$ We always denote by $R_S$ the localization of $R$ at $S$.  A multiplicative set is called regular if it consists of non-zero-divisors. Note that if a multiplicative set $S$ of $R$ is regular, then $R$ can be viewed as a subring of $R_S$ naturally.

Let $R$ be an integral domain with $Q$ its quotient field.  An $R$-module $M$ is said to be \emph{divisible} if $sM=M$ for any non-zero element $s\in R$; and $h$-divisible if it is a quotient module of $Q$-linear space. Trivially,  $h$-divisible modules are always divisible. It is well-known that every divisible $R$-module is  $h$-divisible characterizes Matlis domains, i.e, domains $R$ satisfying $\pd_RQ\leq 1.$ Recall that an $R$-module $M$ is said to be \emph{weakly cotorsion},  if $\Ext_R^1(Q,M)=0$; and \emph{strongly flat} if $\Ext_R^1(M,N)=0$ for any weakly cotorsion module $N$. In 2002, Bazzoni and Salce \cite{BS02}, resolving an open question proposed by Trlifaj \cite{T00},  
showed that every $R$-module admits an strongly flat cover if and only if $R$ is an almost perfect domain, that is,  $R/I$ is a perfect ring for any non-zero proper ideal $I$ of an integral domain $R$.

In the process of solving Trlifaj's problem, Bazzoni and Salce \cite{BS02} also showed that every $h$-divisible modules admits an strongly flat cover over all integral domains; and every divisible  $R$-module admits a strongly flat cover if and only if the integral domain $R$ is Matlis. The main motivation of this paper is to extend these two results to commutative rings with multiplicative sets, a very extensive generalization. Actually, we obtain that   suppose $R_S$ is a semisimple ring with $S$ a regular multiplicative subset of $R$. Then every $S$-$h$-divisible module admits an $S$-strongly flat cover; and every $S$-divisible module admits an $S$-strongly flat cover if and only if $R$ is an $S$-Matlis ring (see Theorem \ref{main1} and Theorem \ref{main2}). Moreover, we show that the condition that ``$R_S$ is a semisimple ring'' cannot be removed by examples (see Remark \ref{remk1} and Remark \ref{remk2}).

\section{Basic properties of $S$-($h$-)divisible modules}

Let $R$ be a ring and $S$ a  multiplicative subset of $R$.  We say an ideal $I$ of $R$ is an $S$-ideal if $I\cap S\not=\emptyset.$   Recall from \cite{Z24-matlisloc} that  
an $R$-module $M$ is said to be
\begin{enumerate}
	\item \emph{$S$-torsion-free} if $sm=0$ with $s\in S$ and $m\in M$ whence $m=0$;
	\item \emph{$S$-torsion} if for any $m\in M$, there is $s\in S$ such that $sm=0$; 
	\item \emph{$S$-divisible} if $sM=M$ for any $s\in S$;
	\item \emph{$S$-reduced} if it has no $S$-divisible submodule;
	\item \emph{$S$-injective} if $\Ext_R^1(R/I,M)=0$ for any $S$-ideal $I$ of $R$.
\end{enumerate}

When $S$ is regular, every $S$-injective $R$-module is $S$-divisible; and it follows by \cite[Proposition 2.2]{Z24-matlisloc} that every  $R_S$-module is $S$-divisible. Moreover, we have the following result.
\begin{lemma}\label{RSmoduleS-inj}
		Let $R$ be a ring,  $S$ a regular  multiplicative subset of $R$ and $M$ an $R_S$-module. Then $M$ is $S$-injective. 
\end{lemma}
\begin{proof} Let $I$ be an $S$-ideal with $s\in I\cap S$, and $f:I\rightarrow M$ be an $R$-homomorphism.
Assume that $f(s)=m\in M$. Define $g:R\rightarrow M$ by $g(r)=s^{-1}rm\in M$. Then $g$ is an $R$-homomorphism. First, we will show $g$ is well-defined. Indeed,
let $s_1\in I\cap S,$ $f(s_1)=m_1$ and  $g_1:R\rightarrow M$ is an $R$-homomorphism induced by $s_1$ as above. Then $$g_1(1)=s_1^{-1}m_1=s_1^{-1}f(s_1)=s^{-1}s_1^{-1}f(ss_1)=s^{-1}s_1^{-1}s_1f(s)=s^{-1}f(s)=g(1).$$
Next, we will show $g$ is an extension of $f$. Indeed, for any $a\in I$, we have $$sf(a)=f(sa)=af(s)=am=s(s^{-1}am).$$ Note that $M$ is $S$-torsion-free. So $f(a)=s^{-1}am=g(a)$. Hence $g$ is an extension of $f$. Consequently, $M$ is $S$-injective.
\end{proof}

Let $R$ be a ring and $S$ a  multiplicative subset of $R$. Let $M$ be an $R$-module. Denote by $$h_S(M)=\sum\limits_{f\in\Hom_R(R_S,M)}\Im(f).$$ 

\begin{definition}\cite{BP19}
	Let $R$ be a ring and $S$ a  multiplicative subset of $R$. 
	An $R$-module $M$ is said to be 
	\begin{enumerate}
		\item $S$-$h$-divisible if $h_S(M)=M$;
		\item $S$-$h$-reduced if $h_S(M)=0$. 
	\end{enumerate}
\end{definition}

It is easy to verify that  an $R$-module $M$  is $S$-$h$-divisible if and only if there is an epimorphism $R_S^{(\kappa)}\twoheadrightarrow M$; and $M$  is $S$-$h$-reduced if and only if $\Hom_R(R_S,M)=0.$ So $S$-$h$-divisible modules are closed under quotients, while $S$-$h$-reduced modules are closed under submodules. Trivially, if the multiplicative subset  $S$ is regular, then every $S$-injective module is $S$-divisible. Moreover, we have the following result.

\begin{proposition}\label{rinj-hd}
	Let $R$ be a ring and $S$ a regular multiplicative subset of $R$. Then an $R$-module is $S$-$h$-divisible if and only if it is a quotient of  an $S$-injective $R$-module.
\end{proposition}
\begin{proof} Let $M$ be an  $S$-$h$-divisible $R$-module. Then there is an epimorphism $R_S^{(\kappa)}\twoheadrightarrow M$. Note that  $R_S^{(\kappa)}$ is  $S$-injective by Lemma \ref{RSmoduleS-inj}. So $M$ is a quotient of  an $S$-injective $R$-module. 
	
On the other hand,  it follows by \cite[Proposition 2.7]{Z24-matlisloc} that every $S$-injective $R$-module is $S$-$h$-divisible, and the class of $S$-$h$-divisible $R$-modules is closed under quotient modules. \end{proof}

The following example shows that $S$-$h$-divisible modules may be not a quotient of an injective module.
\begin{example} Let $R:=\mathbb{Z}(+)\mathbb{Q}/\mathbb{Z}$ be the trivial extension of $\mathbb{Z}$ with $\mathbb{Q}/\mathbb{Z}$. Let $S$ be the set of all non-zero-divisors of $R$. Then $R$ is a total ring of quotients, and so $R$ itself is $S$-$h$-divisible. However, $R$ is not a quotient of an injective $R$-module.	Indeed, suppose there is an exact sequence $0\rightarrow K\rightarrow E\rightarrow R\rightarrow0$ with $E$ injective. Then $R$ is self-injective ring. However, the $R$-homomorphism $f:\langle 2\rangle(+)\mathbb{Q}/\mathbb{Z}\rightarrow R$ with $(2n,\frac{b}{a}+\mathbb{Z})\mapsto (n,\frac{b}{2a}+\mathbb{Z})$ can not be extended to $R$, which is a contradiction.	
\end{example}





Let $M$ be an $R$-module, and $\mathscr{F}$ a class of $R$ modules. Recall from \cite{gt} that an $R$-homomorphism $f : F(M)\rightarrow M$ with $F(M)\in \mathscr{F}$ is called an  $\mathscr{F}$  precover if for
any $R$-homomorphism from an $R$-module in $\mathscr{F}$ to $M$ factors through $f$. Moreover, $f$ is called an  $\mathscr{F}$  cover if  $f$ is  an  $\mathscr{F}$ precover, and whenever $f$ factors as\ $f=f\circ h$\ with $h:F(M)\rightarrow F(M)$ an endmorphism of $F(M)$, then $h$ must be an automorphism. The notion of  pre(envelopes) can be defined dually.

It follows by \cite[Theorem 4.1]{F23} that all modules admit both divisible and $h$-divisible covers over any commutative ring. Now, we give the $S$-version of this result.
\begin{theorem}
	Let $R$ be a ring and $S$ a multiplicative subset of $R$. Then every $R$-module admits an $S$-$(h$-$)$divisible  cover.
\end{theorem}
\begin{proof} Let $M$ be an $R$-module. Set $d(M)$ be the sum of all $S$-divisible submodules of $M$. We claim the embedding map $i:d(M)\hookrightarrow M$ is the $S$-divisible cover of $M$. Obviously, $d(M)$ is also an $S$-divisible $R$-module. Let $D$ be an $S$-divisible $R$-module and $g:D\rightarrow M$ be an $R$-homomorphism. Consider the following diagram:
		$$\xymatrix@R=25pt@C=40pt{
		D\ar@{.>}[d]_{g'}\ar[rd]^{g}& \\
		d(M)\ar@{^{(}->}[r]^{i}&M. \\
	}$$
Since $g(D)$ is an $S$-divisible $R$-module, then $\Im(g)$ is a submodule of $d(M)$, and hence $g$ factor through $d(M)$. 	

Now, assume $h:d(M)\rightarrow d(M)$ is an $R$-homomorphism satisfying $i\circ h=i:$	
		$$\xymatrix@R=25pt@C=40pt{
	d(M)\ar@{.>}[d]_{h}\ar@{^{(}->}[rd]^{i}& \\
	d(M)\ar@{^{(}->}[r]^{i}&M. \\
}$$	
Then $h$ is a  monomorphism. Note $h$ is also an  epimorphism. Indeed, on contrary, assume $d\in d(M)-\Im(h)$. Then $d=i(d)=i(h(d))\in \Im(h),$ which is a contrary. Hence, $h$ is an automorphism. Consequently, 	$i$ is  an $S$-divisible  cover of $M$. 

The existence of $S$-$h$-divisible  covers can be obtained similarly.
\end{proof}

\section{$S$-strongly flat covers of $S$-($h$-)divisible modules}

	Let $R$ be a ring and $S$ a  multiplicative subset of $R$.  Recall from \cite{LS19} that
	an $R$-module $M$ is said to be 
	\begin{enumerate}
		\item \emph{$S$-weakly cotorsion},  if $\Ext_R^1(R_S,M)=0$;
		\item \emph{$S$-strongly flat} if $\Ext_R^1(M,N)=0$ for any $S$-weakly cotorsion module $N$. 
	\end{enumerate}

It follows by \cite[Lemma 1.2]{BP19} that an $R$-module $F$ is $S$-strongly flat if and only if it is a direct summand of an
$R$-module $G$ for which there exists an exact sequence of $R$-modules
$$0\rightarrow U\rightarrow G\rightarrow V \rightarrow 0$$
where $U$ is a free $R$-module and $V$ is a free $R_S$-module. 
 It follows by \cite[Theorem 5.27,Theorem 6.11]{gt} that every $R$-module has an $S\mbox{-}$weakly cotorsion envelope and a special $S\mbox{-}$strongly flat precover. It was proved in  \cite[Theorem 7.9]{BP19} that a ring $R$ is  $S$-almost perfect (i.e., $R_S$ is a perfect ring and $R/sR$ is a perfect ring for any $s\in S$) if and only if every $R$-module admits an  $S\mbox{-}$strongly flat cover.

It was proved in \cite[Theorem 3.1, Corollary 3.2]{BS02} that every $h$-divisible $R$-module admits a strongly flat  cover over all integral domains; and  every divisible  $R$-module admits a strongly flat cover if and only if the integral domain $R$ is Matlis. In general, a natural question is that 
\begin{question} When every $S$-$(h$-$)$divisible module admits an $S\mbox{-}$strongly flat cover?	
\end{question}
In the final of this section, we will ask this question in the case that $R_S$ is a semisimple ring.

\begin{lemma}\label{dirsumad}\cite[Lemma 2]{K52} Let $M$ be any module, $T$ a submodule of $M$, $M/T$ a direct sum of modules $U_i$, and $T_i$ the inverse image in $M$ of $U_i$. Suppose $T$ is a direct summand of each $T_i$. Then $T$ is a direct summand of $M$.
\end{lemma}

\begin{lemma}\label{claim}
	Let $N$ be an $S$-torsion-free and $S$-divisible $R$-module. Then every
	$R$-homomorphism $\alpha:R_S\longrightarrow N$ is automatically an
	$R_S$-homomorphism.
\end{lemma}

\begin{proof}
	Let $\alpha(1)=x$. For $r\in R$ and $s\in S$, we have
	\[
	s\alpha(r/s)=\alpha(r)=rx .
	\]
	Since $N$ is $S$-divisible, there exists $y\in N$ such that
	$sy=rx$. The element $y$ is unique because $N$ is $S$-torsion-free.
	Consequently,
	\[
	\alpha(r/s)=y=(r/s)x=(r/s)\alpha(1).
	\]
	Thus $\alpha$ is $R_S$-linear.
\end{proof}
\begin{proposition}\label{tor-sub-h-div}		Let $R$ be a ring and $S$ a regular multiplicative subset of $R$.  Suppose $R_S$ is a semisimple ring. Then the $S$-torsion submodule of $S$-$h$-divisible module is a summand.
\end{proposition}

\begin{proof}
	Let $M$ be an $S$-$h$-divisible module and let
	\[
	T=M_T=\{m\in M\mid sm=0\text{ for some }s\in S\}.
	\]
	Since $h_S(M)=M$, there exists a cardinal $\kappa$ and an epimorphism
	\[
	f:R_S^{(\kappa)}\longrightarrow M .
	\]
	
	Let $\pi:M\longrightarrow M/T$ be the canonical map. The quotient $M/T$
	is $S$-torsion-free. Moreover, it is $S$-divisible, hence it has a natural
	$R_S$-module structure.
	
	Consider
	\[
	\pi f:R_S^{(\kappa)}\longrightarrow M/T .
	\]
	The map $f$ is obtained as a direct sum of maps
	$R_S\longrightarrow M$. After composing with $\pi$, each component
	\[
	R_S\longrightarrow M/T
	\]
	is $R_S$-linear by the Lemma \ref{claim}. Hence $\pi f$ is an epimorphism of
	$R_S$-modules.
	
	Because $R_S$ is semisimple, every epimorphism of $R_S$-modules splits.
	Therefore there exists an $R_S$-homomorphism
	\[
	\sigma:M/T\longrightarrow R_S^{(\kappa)}
	\]
	such that
	\[
	(\pi f)\sigma=1_{M/T}.
	\]
	
	Put
	\[
	\varphi=f\sigma:M/T\longrightarrow M.
	\]
	Then
	\[
	\pi\varphi=1_{M/T}.
	\]
	Let $U=\operatorname{Im}(\varphi)$. For every $m\in M$,
	\[
	\pi(m-\varphi(\pi(m)))=0,
	\]
	so
	\[
	m-\varphi(\pi(m))\in T.
	\]
	Hence $M=T+U$.
	
	If $u\in T\cap U$, write $u=\varphi(x)$. Since $u\in T=\ker\pi$,
	\[
	0=\pi(u)=\pi\varphi(x)=x.
	\]
	Thus $u=0$, and therefore $T\cap U=0$.
	
	Consequently,
	\[
	M=T\oplus U .
	\]
	Hence the $S$-torsion submodule of $M$ is a direct summand.
\end{proof}
Note that the above Proposition \ref{tor-sub-h-div} may be incorrect when $R_S$ is not semisimple.
\begin{example}
	Let $R:=\mathbb{Z}(+)\mathbb{Q}$ be the trivial extension of $\mathbb{Z}$ with $\mathbb{Q}$, and $S$ the set of all non-zero-divisors of $R$. Then the  ring of quotients of $R$ is $Q=\mathbb{Q}(+)\mathbb{Q}$. So its quotient $M=\mathbb{Q}(+)\mathbb{Q}/0(+)\mathbb{Z}\cong\mathbb{Q}(+)\mathbb{Q}/\mathbb{Z}$  is $h$-divisible. Note that the torsion submodule of $M$ is $T=0(+)\mathbb{Q}/\mathbb{Z}$. Claim that $T$ is not a direct summand of $M$. Indeed, on contrary, suppose $\pi$ is the retraction of the embedding $i:T\rightarrow M.$  Then $\pi((a,\frac{y}{x}+\mathbb{Z}))=(0,\frac{y}{x}+\mathbb{Z})$ for any $(a,\frac{y}{x}+\mathbb{Z})\in T$. However, $\pi((b,\frac{t}{x})(a,\frac{y}{x}+\mathbb{Z}))\not=(b,\frac{t}{x})\pi((a,\frac{y}{x}+\mathbb{Z}))$ when $a\not=1.$ Hence $\pi$ is not an $R$-homomorphism, which is a contradiction. 
\end{example}
\begin{lemma}\label{torsion}
Suppose $T$ is an $S$-torsion $R$-module. Then $\Hom_R(T, M)$ is  $S$-weakly cotorsion and $S$-$h$-reduced for every $R$-module $M$.
\end{lemma}

\begin{proof} Suppose  $T$ is an  $S$-torsion $R$-module.Then $R_S \otimes_R T=0.$
 It follows by $$\Hom_R(R_S, \Hom_R(T, M)) \cong \Hom_R(R_S \otimes_R T, M)=0$$ that  $\Hom_R(T, M)$ is $S$-$h$-reduced. Since $\Tor_R(R_S,T)=0$, it follows by \cite[Lemma 2.2]{FL09} that the natural homomorphism $\Ext_R^1(R_S,\Hom_R(T, M))\rightarrow \Ext_R^1(R_S\otimes_RT, M)=0$ is a monomorphism. Hence $\Ext_R^1(R_S,\Hom_R(T, M))=0,$ that is, $\Hom_R(T, M)$ is $S$-weakly cotorsion.
\end{proof}

\begin{lemma} \label{ssfpre}
 Let
	\[
	0 \longrightarrow C \hookrightarrow M \xrightarrow{f} A \longrightarrow 0
	\]
	be a special $S$-strongly flat-precover of $A$ and let $h$ be an endomorphism of $M$ such that $f = f\circ h$. Then $h(M) + C = M$, $\operatorname{Ker} h \leqslant C$ and $C \cap h(M) = h(C)$.
\end{lemma}
\begin{proof}
	It follows by \cite[Lemma 3.3]{BP19}.
\end{proof}

It was proved in \cite[Theorem 3.1]{BS02} that every $h$-divisible $R$-module admits a strongly flat  cover over all integral domains.
\begin{theorem}\label{main1}
	Let $R$ be a ring and $S$ a regular multiplicative subset of $R$.  Suppose $R_S$ is a semisimple ring. Then every $S$-$h$-divisible module admits an $S$-strongly flat cover.
\end{theorem}
\begin{proof}
It follows by Proposition  \ref{tor-sub-h-div} that the $S$-torsion submodule of an $S$-$h$-divisible module splits. So, by \cite[Proposition 5.5.4]{EJ11}, we only need to consider the  $S$-torsion $S$-$h$-divisible module cases. Let $D$ be an $S$-torsion $S$-$h$-divisible module. Then $h_S(D)=D$, and so there exists an exact sequence
$$	0 \longrightarrow \operatorname{Hom}_R(K, D) \longrightarrow \operatorname{Hom}_R(R_S, D) \xrightarrow{f} D \longrightarrow 0,
$$
where $K=R_S/R.$	Denote by $C$ the first term and by $M$ the middle term. Then $M$ is $S$-torsion-free and $S$-divisible, so is an $R_S$-module, and so hence $S$-strongly flat as $R_S$ is semisimple. It follows by Lemma \ref{torsion} that $C$ is $S$-weakly cotorsion and $S$-$h$-reduced. Hence $f$ is a special $S$-strongly flat precover of $D$.

 Moreover, we will show  that  $f$ is an $S$-strongly flat cover of $D$. Indeed, let $h$ be an endomorphism of $M$ such that $f\circ h = f$. We have to show that $h$ is an automorphism of $M$. Note that $\Ker(h) = 0$. Indeed, $\Ker(h)$ is $S$-torsion-free and $S$-divisible as $M$ is an $R_S$-module. And by Lemma \ref{ssfpre}, $\Ker(h) \subseteq C$ which is $S$-$h$-reduced. Hence $\Ker(h) = 0$.  The image of $h$ is $S$-torsion-free and $S$-divisible. Thus $M = h(M) \oplus M_1$ for some $S$-torsion-free $S$-divisible submodule $M_1$. Let $\pi$ be the projection of $M$ onto the summand $M_1$. By Lemma \ref{ssfpre}, it follows that the restriction of $\pi$ to $C$ is surjective and that its kernel $h(M) \cap C = h(C)$ is $S$-weakly cotorsion, since it is isomorphic to $C$. Thus we can consider the exact sequence
$$
	0 \longrightarrow h(M) \cap C \longrightarrow C \xrightarrow{\pi|_C} M_1 \longrightarrow 0.
$$
	Applying the functor $\operatorname{Hom}_R(R_S, -)$ we obtain the exact sequence
$$
	0 \longrightarrow \operatorname{Hom}_R(R_S, M_1) \cong M_1 \longrightarrow \operatorname{Ext}^1_R(R_S, h(M) \cap C) = 0
$$
	where the left $\operatorname{Ext}$ vanishes, since as we noted $h(M) \cap C$ is $S$-weakly cotorsion. Thus we conclude that $M_1 = 0$ and $h(M) = M$, hence $h$ is an automorphism of $M$.
\end{proof}

\begin{remark}\label{remk1}
	Note that the condition that $R_S$ is semisimple can not be removed in Theorem \ref{main1}. Indeed, let $R$ be a non-semisimple total ring of quotients, and $S$ be the set of all non-zero-divisors of $R$. Then all $R$-modules are $S$-$h$-divisible modules; and $S$-strongly flat modules are exactly projective modules.  However,  every $R$-module  does not admit a projective cover over non-semisimple rings.
\end{remark}

\begin{lemma}\label{s-d-ssf} 	Let $R$ be a ring and $S$ a regular multiplicative subset of $R$.
Then every $S$-divisible $S$-strongly flat $R$-module is $S$-$h$-divisible.
\end{lemma}
\begin{proof} Let $M$ be an $S$-divisible $S$-strongly flat $R$-module. Then $M$ is a direct summand of an
	$R$-module $G$ for which there exists an exact sequence of $R$-modules
	$$0\rightarrow U\rightarrow G\rightarrow V \rightarrow 0$$
	where $U$ is a free $R$-module and $V$ is a free $R_S$-module. Hence $M$ is $S$-torsion free. Since $M$ is $S$-divisible, $M$ is an $R_S$-module by \cite[Proposition 2.2]{Z24-matlisloc}. Hence $M$ is $S$-$h$-divisible by Lemma  \ref{RSmoduleS-inj} and Proposition \ref{rinj-hd}.
\end{proof}

Recall that a ring is called an $S$-Matlis ring if $\pd_RR_S\leq 1$. It was proved in \cite[Corollary 3.2]{BS02} that an integral domain is Matlis if and only if every divisible module admits a strongly flat cover.

\begin{lemma}\label{sdshd}\cite[Theorem 1.1]{AHT05} Let $R$ be a ring and $S$ a regular multiplicative subset of $R$.  Then $R$ is an $S$-Matlis ring if and only if every $S$-divisible $R$-module is $S$-$h$-divisible. 
\end{lemma}

\begin{proposition} \label{sd-ssfc}	Let $R$ be a ring and $S$ a regular multiplicative subset of $R$. If  every $S$-divisible module admits an $S$-strongly flat cover, then $R$ is an $S$-Matlis ring.
\end{proposition}
\begin{proof}
	Let $M$ be an $S$-divisible $R$-module, and $0\rightarrow K\rightarrow F\xrightarrow{f} M\rightarrow 0$ be an exact sequence with $f$ the $S$-strongly flat cover of $M$. Then $K$ is  $S$-weakly cotorsion.
	
First we show that the $S$-strongly flat module $F$ is $S$-divisible. As $M$ is $S$-divisible,  we have $sF + K = F$ for any $s \in S$. Considering the exact sequence
$$
	0 \to sF \cap K \to K \to F/sF \to 0,
$$
we have an sequence  $$0=\Hom_R(R_S,F/sF)\rightarrow \Ext_R^1(R_S,sF\cap K)\rightarrow \Ext_R^1(R_S,K)=0.$$ So $\Ext_R^1(R_S,sF\cap K)=0$, that is, $sF \cap K$ is $S$-weakly cotorsion. Therefore, there exists a map $g: F \to sF$ making  the following diagram commute:	
$$\xymatrix@C=40pt@R=30pt{
0\ar[r]&	K \ar[r]\ar[d] & F\ar[r]^f\ar@{.>}[d]^{g}&M \ar[r]\ar@{=}[d]& 0\\
0\ar[r]&rF\cap K \ar[d] \ar[r] & rF \ar[d]^{i}\ar[r]&M \ar@{=}[d]\ar[r] &0 \\
0\ar[r]&K  \ar[r] & F \ar[r]^f &M \ar[r] &0
}$$		
with the embedding map $i: sF \to F$. The diagram shows that $f = f\circ i\circ g$, whence $i\circ g$ is an automorphism of $F$  by the cover property of $F$. Consequently, $i$ is an epimorphism, and $sF = F$, as claimed. Thus the $S$-strongly flat module $F$ is $S$-divisible, and hence is also $S$-$h$-divisible by Lemma \ref{s-d-ssf}. Hence, every $S$-divisible $R$-module is $S$-$h$-divisible.  Consequently, $R$ is an $S$-Matlis ring by Lemma \ref{sdshd}.
\end{proof}

\begin{theorem}\label{main2}
	Let $R$ be a ring and $S$ a regular multiplicative subset of $R$.  Suppose $R_S$ is a semisimple ring.  Then every $S$-divisible module admits an $S$-strongly flat cover if and only if $R$ is an $S$-Matlis ring.
\end{theorem}
\begin{proof} Combine Theorem \ref{main1}, Lemma \ref{sdshd}, and Proposition \ref{sd-ssfc}.
\end{proof}
\begin{remark}\label{remk2}
	Note that the condition that $R_S$ is semisimple can not be removed in Theorem \ref{main2} similar to Remark \ref{remk1}. Indeed, let $R$ be a non-semisimple total ring of quotients, and $S$ be the set of all non-zero-divisors of $R$. Then all $R$-modules are $S$-divisible modules; and $S$-strongly flat modules are exactly projective modules.  However,  every $R$-module does not admit a projective cover over non-semisimple rings.
\end{remark}

\textbf{Conflict of interest.} The author states that there is no conflict of interest.

\end{document}